\documentclass[12pt]{article}
\usepackage{amsmath,amsthm,amsfonts,amssymb,amscd}

\allowdisplaybreaks[4]
\newtheorem {Lemma}{Lemma}[section]
\newtheorem {Theorem} {Theorem}[section]

\newenvironment {Proof} {\noindent {\textit{Proof of Theorem~\ref{conn-graph}.}}}{\quad $\square$\par\vspace{3mm}}

= 0.0in \topmargin = 0.3in \oddsidemargin=0.1in
\begin{document}

\title{Graphs characterized by the second distance eigenvalue}

\author{Rundan Xing, Bo Zhou\footnote{Corresponding author. E-mail: zhoubo@scnu.edu.cn}\\
Department of Mathematics, South China Normal University,\\
Guangzhou 510631, P.R. China}

\date{}
\maketitle

\begin{abstract}
We characterize all connected graphs with second distance eigenvalue less than $-0.5858$. 
\end{abstract}

\section{Introduction}

We consider simple undirected graphs. Let $G$ be a connected graph
with vertex set $V(G)=\{v_1,\dots,v_n\}$. For
$1\le i,j\le n$, the distance between vertices $v_i$ and $v_j$ in $G$,
denoted by $d_G(v_i,v_j)$ or simply $d_{v_iv_j}$, is the length of a
shortest path connecting them in $G$. The distance matrix of $G$ is
the $n\times n$ matrix $D(G)=(d_{v_iv_j})$. Since
$D(G)$ is symmetric, the eigenvalues of $D(G)$ are
all real numbers. The distance eigenvalues of $G$, denoted by
$\lambda_1(G),\dots,\lambda_n(G)$, are the eigenvalues of
$D(G)$, arranged in non-increasing order. For $1\le k\le n$, we call $\lambda_k(G)$ the $k$th distance eigenvalue of $G$.


The study of distance eigenvalues dates back to the classical work
of Graham and Pollack~\cite{GP}, Edelberg et al.~\cite{EGG} and Graham
and Lov\'asz~\cite{GL} in 1970s. Merris~\cite{Me90} studied the
relations between the distance eigenvalues and the Laplacian
eigenvalues of trees. The first distance eigenvalue has received much attention. Ruzieh
and Powers~\cite{SuPo90} 
showed that the path $P_n$ is the unique $n$-vertex connected graph
with maximal first distance eigenvalue, while the complete graph
$K_n$ is the unique $n$-vertex connected graph with minimal first
distance eigenvalue. Among others, Stevanovi\'c and Ili\'c~\cite{SI}
showed that the star $S_n$ is the unique $n$-vertex tree with
minimal first distance eigenvalue. The extremal graphs with maximal
or minimal first distance eigenvalues may be found in, e.g.,
\cite{BNP,NP,WZ,XZD,YW}.
The last (least) distance eigenvalue has also received some
attention, see~\cite{LZ,Yu}.


%
%
%
%

Let $G_1$ and $G_2$ be two vertex-disjoint graphs. $G_1\cup G_2$
denotes the vertex-disjoint union of $G_1$ and $G_2$, and $G_1\vee
G_2$ denotes the graph obtained from $G_1\cup G_2$ by joining each
vertex of $G_1$ and each vertex of $G_2$ using an edge.

In this paper, we characterize all connected graphs with second
distance eigenvalue less than $-0.5858$. We prove the following
result.

\begin{Theorem}\label{conn-graph}
Let $G$ be a connected graph with at least two vertices. Then
$\lambda_2(G)\in(-\infty,-0.5858)$ if and only if $G\cong K_n$ for
some $n\ge 2$, or $G\cong K_1\vee(K_{n_1}\cup K_{n_2})$ for some
$n_1,n_2\ge 1$, or $G\cong K_1\vee(K_{n_1}\cup K_{n_2}\cup K_{n_3})$
for some  $n_1,n_2,n_3\ge 1$, or $G\cong K_1\vee(K_{n_1}\cup
K_{n_2}\cup K_{n_3}\cup K_{n_4})$ for some $n_1,n_2,n_3, n_4$ such
that $1\le n_1\le n_2\le n_3\le n_4$ and one of the following items
holds$:$

(i) $n_1=n_2=1$ and $1\le n_3\le 873$;

(ii) $n_1=n_2=1$, $n_3\ge 874$, and $r(-0.5858)<0$, where
\begin{eqnarray*}
   r(\lambda)
&=&
   \lambda^4-(n_3+n_4)\lambda^3-(3n_3n_4+8n_3+8n_4+5)\lambda^2\\
& &
   -(12n_3n_4+11n_3+11n_4+6)\lambda-(6n_3n_4+4n_3+4n_4+2);
\end{eqnarray*}

(iii) $n_1=1$, $n_2=2$, and $n_3=2$, or $3=n_3\le n_4\le 870$, or $4=n_3\le n_4\le 14$, or $5=n_3\le n_4\le 8$, or $n_3=n_4=6$;

(iv) $n_1=1$, $n_2=3$, and $3=n_3\le n_4\le 7$ or $n_3=n_4=4$;

(v) $n_1=n_2=n_3=2$ and $2\le n_4\le 5$.
\end{Theorem}

\section{Proof of Theorem~\ref{conn-graph}}

For an $n\times n$ symmetric matrix $M$, let
$\mu_1(M),\dots,\mu_n(M)$ be the eigenvalues of $M$, arranged in non-increasing order. Let $A$
be an $n\times n$ symmetric matrix, and $B$ an $m\times m$ principal
submatrix of $A$. The interlacing theorem \cite[pp.~185--186]{HJ}
states that $\mu_{n-m+i}(A)\le\mu_i(B)\le\mu_i(A)$ for $1\le i\le
m$.

Let $G$ be an $n$-vertex connected graph, and $H$ an $m$-vertex
induced subgraph of $G$, where $m\ge 2$. If $H$ is connected and
$d_H(u,v)=d_G(u,v)$ for all $\{u,v\}\subseteq V(H)$, then write
$H\trianglelefteq G$. Obviously, if $H$ is of diameter two, then
$H\trianglelefteq G$. If $H\trianglelefteq G$, then $D(H)$ is a
principal submatrix of $D(G)$, and thus from the interlacing theorem,
$\lambda_2(G)\ge\lambda_2(H)$ and $\lambda_n(G)\le\lambda_m(H)$.

For integer $n\ge 1$, let $nG$ be the vertex-disjoint union of $n$ copies of graph $G$.

Let $I_n$ be the $n\times n$ identity matrix, and $J_{m\times n}$
the $m\times n$ all-one matrix. For convenience, let $J_n=J_{n\times
n}$ and  $\mathbf{1}_n=J_{n\times 1}$.

\begin{Lemma}\label{Kn12-Kn123}
For integers $n_1,n_2,n_3\ge 1$, $\lambda_2(K_1\vee(K_{n_1}\cup K_{n_2}))\in(-1,-0.5858)$ and $\lambda_2(K_1\vee(K_{n_1}\cup K_{n_2}\cup K_{n_3}))\in(-1,-0.5858)$.
\end{Lemma}

\begin{proof}
Let $G_1=K_1\vee(K_{n_1}\cup K_{n_2})$.
We have, with respect to the partition $V(G_1)=V(K_1)\cup V(K_{n_1})\cup V(K_{n_2})$, that
\[
D(G_1)=\left[\begin{array}{ccc}
  0 & \mathbf{1}_{n_1}^{\top} & \mathbf{1}_{n_2}^{\top}\\
   \mathbf{1}_{n_1} & J_{n_1}-I_{n_1} & 2J_{n_1\times n_2}\\
   \mathbf{1}_{n_2} & 2J_{n_2\times n_1} & J_{n_2}-I_{n_2}
   \end{array}\right].
\]
Then the characteristic polynomial of $D(G_1)$ is
\begin{eqnarray*}
   \det(\lambda I_{n_1+n_2+1}-D(G_1))
&=&
   \left|\begin{array}{ccc}
   \lambda & -\mathbf{1}_{n_1}^{\top} & -\mathbf{1}_{n_2}^{\top}\\
   -\mathbf{1}_{n_1} & (\lambda+1)I_{n_1}-J_{n_1} & -2J_{n_1\times n_2}\\
   -\mathbf{1}_{n_2} & -2J_{n_2\times n_1} & (\lambda+1)I_{n_2}-J_{n_2}
   \end{array}\right|\\
&=&
   (\lambda+1)^{n_1+n_2-2}
   \left|\begin{array}{ccc}
   \lambda+n_2+1 & -(\lambda+n_1+1) & 0\\
   -n_2 & \lambda-n_1 & -\lambda\\
   -n_2 & 0 & \lambda+1
   \end{array}\right|\\
&=&
   (\lambda+1)^{n_1+n_2-2}f(\lambda),
\end{eqnarray*}
where
\[
f(\lambda)=\lambda^3-(n_1+n_2-2)\lambda^2-(3n_1n_2+2n_1+2n_2-1)\lambda-2n_1n_2-n_1-n_2.
\]
Obviously, $K_1\vee 2K_1\cong P_3\trianglelefteq G_1$, implying that $\lambda_2(G_1)\ge\lambda_2(P_3)=1-\sqrt{3}>-1$, and thus
$\lambda_2(G_1)$ is the second largest root of the equation
$f(\lambda)=0$. By direct check, we have $f(-1)=n_1n_2>0$ and
$f(-0.5858)=-0.2426n_1n_2-0.1716(n_1+n_2)-0.1005<0$,
and thus $\lambda_2(G_1)\in(-1,-0.5858)$.

Now let $G_2=K_1\vee(K_{n_1}\cup K_{n_2}\cup K_{n_3})$ and $n=n_1+n_2+n_3+1$. The characteristic polynomial of $D(G_2)$ is
\begin{eqnarray*}
& &
   \det(\lambda I_n-D(G_2))\\
&=&
   \left|\begin{array}{cccc}
   \lambda & -\mathbf{1}_{n_1}^{\top} & -\mathbf{1}_{n_2}^{\top} & -\mathbf{1}_{n_3}^{\top}\\
   -\mathbf{1}_{n_1} & (\lambda+1)I_{n_1}-J_{n_1} & -2J_{n_1\times n_2} & -2J_{n_1\times n_3}\\
   -\mathbf{1}_{n_2} & -2J_{n_2\times n_1} & (\lambda+1)I_{n_2}-J_{n_2} & -2J_{n_2\times n_3}\\
   -\mathbf{1}_{n_3} & -2J_{n_3\times n_1} & -2J_{n_3\times n_2} & (\lambda+1)I_{n_3}-J_{n_3}\\
   \end{array}\right|\\
&=&
   (\lambda+1)^{n_1+n_2+n_3-3}
   \left|\begin{array}{cccc}
   \lambda+n_1+1 & \lambda+n_2+1 & 0 & 0\\
   0 & -(\lambda+n_2+1) & \lambda+n_3+1 & 0\\
   -(\lambda-n_1) & -n_2 & -n_3 & \lambda\\
   0 & -n_2 & -n_3 & -(\lambda+1)
   \end{array}\right|\\
&=&
   (\lambda+1)^{n-4}g(\lambda),
\end{eqnarray*}
where
\begin{eqnarray*}
   g(\lambda)
&=&
   \lambda^4-(n_1+n_2+n_3-3)\lambda^3-3(n_1n_2+n_1n_3+n_2n_3+n_1+n_2+n_3-1)\lambda^2\\
& &
   -[5(n_1n_2+n_1n_3+n_2n_3+n_1n_2n_3)+3(n_1+n_2+n_3)-1]\lambda\\
& &
   -3n_1n_2n_3-2(n_1n_2+n_1n_3+n_2n_3)-(n_1+n_2+n_3).
\end{eqnarray*}
Obviously, $K_1\vee 3K_1\cong S_4\trianglelefteq G_2$, implying that $\lambda_2(G_2)\ge\lambda_2(S_4)=2-\sqrt{7}>-1$ and $\lambda_n(G_2)\le\lambda_4(S_4)=-2<-1$, and thus $\lambda_2(G_2)$ and $\lambda_n(G_2)$ are
the second largest and the least roots of the equation $g(\lambda)=0$ respectively. By direct
check, $g(-1)=2n_1n_2n_3>0$ and
\begin{eqnarray*}
   g(-0.5858)
&=&
   -0.0710n_1n_2n_3-0.1005(n_1n_2+n_1n_3+n_2n_3)\\
& &
   -0.0711(n_1+n_2+n_3)-0.0416\\
&<&0.
\end{eqnarray*}
Thus we have $\lambda_2(G_2)\in(-1,-0.5858)$.
\end{proof}


\begin{Lemma}\label{Kn1234}
Let $n_1,n_2,n_3,n_4$ be positive integers with $n_1\le n_2\le n_3\le
n_4$. Then $\lambda_2(K_1\vee(K_{n_1}\cup K_{n_2}\cup K_{n_3}\cup
K_{n_4}))\in(-1,-0.5858)$ if and only if one of the following items holds$:$

(i) $n_1=n_2=1$ and $1\le n_3\le 873$;

(ii) $n_1=n_2=1$, $n_3\ge 874$, and $r(-0.5858)<0$, where
\begin{eqnarray*}
   r(\lambda)
&=&
   \lambda^4-(n_3+n_4)\lambda^3-(3n_3n_4+8n_3+8n_4+5)\lambda^2\\
& &
   -(12n_3n_4+11n_3+11n_4+6)\lambda-(6n_3n_4+4n_3+4n_4+2);
\end{eqnarray*}

(iii) $n_1=1$, $n_2=2$, and $n_3=2$, or $3=n_3\le n_4\le 870$, or $4=n_3\le n_4\le 14$, or $5=n_3\le n_4\le 8$, or $n_3=n_4=6$;

(iv) $n_1=1$, $n_2=3$, and $3=n_3\le n_4\le 7$ or $n_3=n_4=4$;

(v) $n_1=n_2=n_3=2$ and $2\le n_4\le 5$.
\end{Lemma}

\begin{proof}
Let $G=K_1\vee(K_{n_1}\cup K_{n_2}\cup K_{n_3}\cup K_{n_4})$ and $n=n_1+n_2+n_3+n_4+1$.
The characteristic polynomial of $D(G)$ is
\begin{eqnarray*}
& &
   \det(\lambda I_n-D(G))\\
&=&
   \left|\begin{array}{ccccc}
   \lambda & -\mathbf{1}_{n_1}^{\top} & -\mathbf{1}_{n_2}^{\top} & -\mathbf{1}_{n_3}^{\top} & -\mathbf{1}_{n_4}^{\top}\\
   -\mathbf{1}_{n_1} & (\lambda+1)I_{n_1}-J_{n_1} & -2J_{n_1\times n_2} & -2J_{n_1\times n_3} & -2J_{n_1\times n_4}\\
   -\mathbf{1}_{n_2} & -2J_{n_2\times n_1} & (\lambda+1)I_{n_2}-J_{n_2} & -2J_{n_2\times n_3} & -2J_{n_2\times n_4}\\
   -\mathbf{1}_{n_3} & -2J_{n_3\times n_1} & -2J_{n_3\times n_2} & (\lambda+1)I_{n_3}-J_{n_3} & -2J_{n_3\times n_4}\\
   -\mathbf{1}_{n_4} & -2J_{n_4\times n_1} & -2J_{n_4\times n_2} & -2J_{n_4\times n_3} & (\lambda+1)I_{n_4}-J_{n_4}\\
   \end{array}\right|\\
&=&
   (\lambda+1)^{n_1+n_2+n_3+n_4-4}\cdot\\
& &
   \left|\begin{array}{ccccc}
   -(\lambda+n_1+1) & \lambda+n_2+1 & 0 & 0 & 0\\
   0 & -(\lambda+n_2+1) & \lambda+n_3+1 & 0 & 0\\
   0 & 0 & -(\lambda+n_3+1) & \lambda+n_4+1 & 0\\
   \lambda-n_1 & -n_2 & -n_3 & -n_4 & \lambda\\
   0 & -n_2 & -n_3 & -n_4 & -(\lambda+1)
   \end{array}\right|\\
&=&
   (\lambda+1)^{n-5}h_{n_1,n_2,n_3,n_4}(\lambda),
\end{eqnarray*}
where
\begin{eqnarray*}
& &
   h_{n_1,n_2,n_3,n_4}(\lambda)\\
&=&
   \lambda^5-(n_1+n_2+n_3+n_4-4)\lambda^4-[3(n_1n_2+n_1n_3+n_1n_4+n_2n_3+n_2n_4+n_3n_4)\\
& &
   +4(n_1+n_2+n_3+n_4)-6]\lambda^3-[5(n_1n_2n_3+n_1n_2n_4+n_1n_3n_4+n_2n_3n_4)\\
& &
   +8(n_1n_2+n_1n_3+n_1n_4+n_2n_3+n_2n_4+n_3n_4)+6(n_1+n_2+n_3+n_4)-4]\lambda^2\\
& &
   -[7n_1n_2n_3n_4+8(n_1n_2n_3+n_1n_2n_4+n_1n_3n_4+n_2n_3n_4)\\
& &
  +7(n_1n_2+n_1n_3+n_1n_4+n_2n_3+n_2n_4+n_3n_4)+4(n_1+n_2+n_3+n_4)-1]\lambda\\
& &
   -4n_1n_2n_3n_4-3(n_1n_2n_3+n_1n_2n_4+n_1n_3n_4+n_2n_3n_4)\\
& &
   -2(n_1n_2+n_1n_3+n_1n_4+n_2n_3+n_2n_4+n_3n_4)-(n_1+n_2+n_3+n_4).
\end{eqnarray*}
Obviously, $K_1\vee 4K_1\cong S_5\trianglelefteq G$, implying that $\lambda_2(G)\ge\lambda_2(S_5)=3-\sqrt{13}>-1$, and thus $\lambda_2(G)$ is the
second largest root of the equation $h_{n_1,n_2,n_3,n_4}(\lambda)=0$.

If $n_1\ge 3$, then $K_1\vee 4K_3\trianglelefteq G$, and thus $\lambda_2(G)\ge\lambda_2(K_1\vee
4K_3)=-0.5830>-0.5858$. 
If $n_2\ge 4$, then $K_1\vee (K_1\cup 3K_4)\trianglelefteq G$, and thus $\lambda_2(G)\ge\lambda_2(K_1\vee (K_1\cup
3K_4))=-0.5855>-0.5858$. 
To obtain the result, we need only to consider the cases when $n_1\le 2$ and $n_2\le 3$.

\noindent {\bf Case 1.} $n_1=1$.

\noindent {\bf Case 1.1.} $n_2=1$. If $n_4=1$, then $G=K_1\vee 4K_1\cong S_5$ with $\lambda_2(G)=3-\sqrt{13}\in(-1,-0.5858)$. Suppose that $n_4\ge 2$. We have $h_{n_1,n_2,n_3,n_4}(\lambda)=h_{1,1,n_3,n_4}(\lambda)=(\lambda+2)r_{1,1,n_3,n_4}(\lambda)$, where
\begin{eqnarray*}
   r_{1,1,n_3,n_4}(\lambda)
&=&
   \lambda^4-(n_3+n_4)\lambda^3-(3n_3n_4+8n_3+8n_4+5)\lambda^2\\
& &
   -(12n_3n_4+11n_3+11n_4+6)\lambda-(6n_3n_4+4n_3+4n_4+2).
\end{eqnarray*}
Then $\det(\lambda
I_n-D(G))=(\lambda+1)^{n-5}(\lambda+2)r_{1,1,n_3,n_4}(\lambda)$. Recall that $\lambda_2(G)>-1$ is the second largest root of $h_{1,1,n_3,n_4}(\lambda)=0$, and thus $\lambda_2(G)$ is also the second largest root of $r_{1,1,n_3,n_4}(\lambda)=0$. Obviously, $K_1\vee(3K_1\cup K_2)\trianglelefteq G$, and then $\lambda_n(G)\le\lambda_6(K_1\vee(3K_1\cup K_2))=-2.6288<-2$, implying that $\lambda_n(G)$ is the least root of $r_{1,1,n_3,n_4}(\lambda)=0$. Note that $\lambda_1(G)>0$ is the largest root of $r_{1,1,n_3,n_4}(\lambda)=0$. Since $r_{1,1,n_3,n_4}(-1)=3n_3n_4>0$, we have $\lambda_2(G)\in (-1,-0.5858)$ if and only if
$r_{1,1,n_3,n_4}(-0.5858)<0$. If $1\le n_3\le 873$, then
\begin{eqnarray*}
   r_{1,1,n_3,n_4}(-0.5858)
&=&
   (1.1508\times 10^{-4}n_3-0.1005)n_4-0.1005n_3-0.0832\\
&<&
   -0.1005n_3-0.0832\\
&<&
   0.
\end{eqnarray*}
Thus $\lambda_2(G)\in(-1,-0.5858)$ if and only if items~(i)
or~(ii) holds.

\noindent {\bf Case 1.2.} $n_2=2$. Then $n_4\ge n_3\ge 2$.

Suppose first that $n_3=2$. If $n_4=2$, then $G=K_1\vee(K_1\cup 3K_2)$ with
$\lambda_2(G)=-0.5925\in(-1,-0.5858)$. Assume that $n_4\ge 3$.
We have
$h_{n_1,n_2,n_3,n_4}(\lambda)=h_{1,2,2,n_4}(\lambda)=(\lambda+3)s_{1,2,2,n_4}(\lambda)$,
where
\[
s_{1,2,2,n_4}(\lambda)=\lambda^4-(n_4+4)\lambda^3-(16n_4+26)\lambda^2-(38n_4+32)\lambda-(17n_4+11).
\]
Then $\det(\lambda
I_n-D(G))=(\lambda+1)^{n-5}(\lambda+3)s_{1,2,2,n_4}(\lambda)$.
Obviously, $K_1\vee(K_1\cup 2K_2\cup K_3)\trianglelefteq G$,
and then $\lambda_n(G)\le\lambda_9(K_1\vee(K_1\cup
2K_2\cup K_3))=-3.6122<-3$, which, together with the fact that
$\lambda_2(G)>-1$, implies that $\lambda_2(G)$ and $\lambda_n(G)$
are the second largest and the least roots of
$s_{1,2,2,n_4}(\lambda)=0$ respectively. By direct check,
$s_{1,2,2,n_4}(-1)=6n_4>0$ and
$s_{1,2,2,n_4}(-0.5858)=-0.0292n_4-0.2547<0$, implying that
$\lambda_2(G)\in (-1,-0.5858)$. 

Suppose that $n_3=3$. We have
\begin{eqnarray*}
   h_{n_1,n_2,n_3,n_4}(\lambda)
&=&
   h_{1,2,3,n_4}(\lambda)\\
&=&
   \lambda^5-(n_4+2)\lambda^4-(22n_4+51)\lambda^3-(109n_4+150)\lambda^2\\
& &
   -(176n_4+148)\lambda-70n_4-46.
\end{eqnarray*}
Recall that $\lambda_2(G)$ is the second largest root of
$h_{1,2,3,n_4}(\lambda)=0$. Since $n_4\ge 3$, we have by direct
check that $h_{1,2,3,n_4}(-3.9)=32.1839n_4-89.9612>0$,
$h_{1,2,3,n_4}(-3)=-10n_4+20<0$, $h_{1,2,3,n_4}(-1)=18n_4>0$, and
$h_{1,2,3,n_4}(0)=-70n_4-46<0$. Thus $\lambda_2(G)\in (-1,0)$, and
$\lambda_2(G)\in (-1,-0.5858)$ if and only if
$h_{1,2,3,n_4}(-0.5858)=9.5128\times 10^{-4}n_4-0.8281<0$, i.e.,
$3\le n_4\le 870$.

Suppose that $n_3=4$. If $n_4\ge 15$, then $K_1\vee (K_1\cup K_2\cup
K_4\cup K_{15})\trianglelefteq G$, and thus
$\lambda_2(G)\ge\lambda_2(K_1\vee (K_1\cup K_2\cup K_4\cup
K_{15}))=-0.58577>-0.5858$. 
Hence, together with Tabel~\ref{tab1}, we have  $\lambda_2(G)\in(-1,-0.5858)$  if and only if
$4\le n_4\le 14$.

\begin{table}[h]
\small \caption{The second distance eigenvalue of $K_1\vee (K_1\cup K_2\cup K_4\cup K_{n_4})$ for $4\le n_4\le 14$.}\label{tab1}
\bigskip
\renewcommand{\arraystretch}{1.6}
\centering
\begin{tabular}{|c|c|c|c|c|c|c|}
\noalign{\hrule height 0.5pt}
$n_4$ & $4$ & $5$ & $6$ & $7$ & $8$ & $9$\\
\noalign{\hrule height 0.4pt}
\hline $\lambda_2$ & $-0.5877$ & $-0.5872$ & $-0.5869$ & $-0.5866$ & $-0.5864$ & $-0.5863$\\
\noalign{\hrule height 0.4pt}
$n_4$ & $10$ & $11$ & $12$ & $13$ & $14$ &\\
\noalign{\hrule height 0.4pt}
\hline $\lambda_2$ & $-0.5861$ & $-0.58604$ & $-0.58595$ & $-0.58588$ & $-0.58582$ &\\
\noalign{\hrule height 0.4pt} \hline
\end{tabular}
\end{table}

Suppose that $n_3=5$. If $n_4\ge 9$, then $K_1\vee (K_1\cup K_2\cup
K_5\cup K_9)\trianglelefteq G$, and thus
$\lambda_2(G)\ge\lambda_2(K_1\vee (K_1\cup K_2\cup K_5\cup
K_9))=-0.58576>-0.5858$. 
Hence, together with Tabel~\ref{tab2}, we have $\lambda_2(G)\in(-1,-0.5858)$  if and only if
$5\le n_4\le 8$.

\begin{table}[h]
\small \caption{The second distance eigenvalue of $K_1\vee (K_1\cup K_2\cup K_5\cup K_{n_4})$ for $5\le n_4\le 8$.}\label{tab2}
\bigskip
\renewcommand{\arraystretch}{1.6}
\centering
\begin{tabular}{|c|c|c|c|c|}
\noalign{\hrule height 0.5pt}
$n_4$ & $5$ & $6$ & $7$ & $8$\\
\noalign{\hrule height 0.4pt}
\hline $\lambda_2$ & $-0.5867$ & $-0.5864$ & $-0.5861$ & $-0.5859$\\
\noalign{\hrule height 0.4pt} \hline
\end{tabular}
\end{table}

Suppose finally that $n_3\ge 6$. If $n_4\ge 7$, then $K_1\vee
(K_1\cup K_2\cup K_6\cup K_7)\trianglelefteq G$, and thus $\lambda_2(G)\ge\lambda_2(K_1\vee (K_1\cup
K_2\cup K_6\cup K_7))=-0.58576>-0.5858$.  By direct check, $\lambda_2(K_1\vee (K_1\cup K_2\cup
2K_6))=-0.5860\in(-1,-0.5858)$. Thus $\lambda_2(G)\in(-1,-0.5858)$  if and only if
$n_3=n_4=6$.

From this case, we have $\lambda_2(G)\in(-1,-0.5858)$ if and only if item~(iii) holds.

\noindent {\bf Case 1.3.} $n_2=3$.

Suppose first that $n_3=3$. If
$n_4\ge 8$, then $K_1\vee (K_1\cup 2K_3\cup K_8)\trianglelefteq G$,
and thus $\lambda_2(G)\ge\lambda_2(K_1\vee
(K_1\cup 2K_3\cup K_8))=-0.58576>-0.5858$. 
Hence, together with Table~\ref{tab3}, we have
$\lambda_2(G)\in(-1,-0.5858)$ if and only if $3\le n_4\le 7$.

\begin{table}[h]
\small \caption{The second distance eigenvalue of $K_1\vee (K_1\cup 2K_3\cup K_{n_4})$ for $3\le n_4\le 7$.}\label{tab3}
\bigskip
\renewcommand{\arraystretch}{1.6}
\centering
\begin{tabular}{|c|c|c|c|c|c|}
\noalign{\hrule height 0.5pt}
$n_4$ & $3$ & $4$ & $5$ & $6$ & $7$\\
\noalign{\hrule height 0.4pt}
\hline $\lambda_2$ & $-0.5878$ & $-0.5870$ & $-0.5865$ & $-0.5862$ & $-0.5859$\\
\noalign{\hrule height 0.4pt} \hline
\end{tabular}
\end{table}

Now suppose
that $n_3\ge 4$. If $n_4\ge 5$, then $K_1\vee (K_1\cup K_3\cup
K_4\cup K_5)\trianglelefteq G$, and thus
$\lambda_2(G)\ge\lambda_2(K_1\vee (K_1\cup K_3\cup K_4\cup
K_5))=-0.58575>-0.5858$. 
By direct check, $\lambda_2(K_1\vee (K_1\cup K_3\cup
2K_4))=-0.5862\in(-1,-0.5858)$. Thus $\lambda_2(G)\in(-1,-0.5858)$ if and only if $n_3=n_4=4$.

From this case, we have $\lambda_2(G)\in(-1,-0.5858)$ if and only if item~(iv) holds.


\noindent {\bf Case 2.} $n_1=2$. If
$n_3\ge 3$, then $K_1\vee (2K_2\cup 2K_3)\trianglelefteq G$, and
thus $\lambda_2(G)\ge\lambda_2(K_1\vee(2K_2\cup 2K_3))=\sqrt{2}-2>-0.5858$.
If $n_4\ge 6$, then $K_1\vee (3K_2\cup
K_6)\trianglelefteq G$, and thus
$\lambda_2(G)\ge\lambda_2(K_1\vee (3K_2\cup K_6))=-0.5856>-0.5858$.
Hence, together with Tabel~\ref{tab4}, we have $\lambda_2(G)\in(-1,-0.5858)$ if and only if item~(v) holds.
\end{proof}

\begin{table}[h]
\small \caption{The second distance eigenvalue of $K_1\vee (3K_2\cup K_{n_4})$ for $2\le n_4\le 5$.}\label{tab4}
\bigskip
\renewcommand{\arraystretch}{1.6}
\centering
\begin{tabular}{|c|c|c|c|c|}
\noalign{\hrule height 0.5pt}
$n_4$ & $2$ & $3$ & $4$ & $5$\\
\noalign{\hrule height 0.4pt}
\hline $\lambda_2$ & $-0.5887$ & $-0.5872$ & $-0.5864$ & $-0.5859$\\
\noalign{\hrule height 0.4pt} \hline
\end{tabular}
\end{table}

%


Let $C_n$ be the cycle on $n$ vertices.

Let $G$ be a graph. Let $|G|=|V(G)|$. For $e\in E(G)$, let $G-e$ be the graph obtained from $G$ by deleting the edge $e$.
Let $\overline{G}$ be the complement of $G$.

\begin{Lemma}\label{com}\cite[p.~10]{Me2001}
Let $G$ be a graph with at least two vertices. If both $G$ and $\overline{G}$ are connected, then $G$ contains an induced subgraph isomorphic to $P_4$.
\end{Lemma}

Now we are ready to prove our main result.

\begin{Proof}
If $G$ is not  complete, then 
$P_3\trianglelefteq G$, and thus $\lambda_2(G)\ge
\lambda_2(P_3)=1-\sqrt{3}>-1$. For $n\ge 2$,
$D(K_n)=J_n-I_n$, and thus $\lambda_2(K_n)=-1$. Thus
$\lambda_2(G)\in(-\infty, -1]$ if and only if $G\cong K_n$ for some
$n\ge 2$.

Suppose that $G$ is a connected graph with
$\lambda_2(G)\in(-1,-0.5858)$.

Suppose that $\overline{G}$ is connected. By Lemma~\ref{com}, $G$
contains an induced subgraph isomorphic to $P_4$. Then
$P_4\trianglelefteq G$ or $C_5\trianglelefteq G$, and thus $\lambda_2(G)\ge
\min\{\lambda_2(P_4),\lambda_2(C_5)\}=\lambda_2(P_4)=-0.58579>-0.5858$,
a contradiction. It follows that $\overline{G}$ is disconnected. Let
$H_1,\dots,H_k$ be the (connected) components of $\overline{G}$,
where $k\ge 2$. Let $G_i=\overline{H_i}$, where $1\le i\le k$.
Obviously, $G_i$ is an induced subgraph of $G$ for $1\le i\le k$,
and $G\cong\vee_{i=1}^k G_i$.

\noindent {\bf Claim 1.} If  $|G_i|\ge 2$, then $G_i$ is disconnected in $G$, where $1\le i\le k$.

Suppose that $|G_i|\ge 2$ and $G_i$ is a connected subgraph of $G$ for some $i$
with $1\le i\le k$. Note that $H_i=\overline{G_i}$ is connected in
$\overline{G}$. By Lemma~\ref{com}, $G_i$ contains an induced
subgraph isomorphic to $P_4$. Since $k\ge 2$, we have
$K_1\vee P_4\trianglelefteq G$, and thus
$\lambda_2(G)\ge\lambda_2(K_1\vee P_4)=-0.3820>-0.5858$, a
contradiction. This proves Claim $1$.

If $|G_i|\ge 2$ for each $i$ with $1\le i\le k$, then by Claim $1$,
$G_i$ is disconnected in $G$, which, by noting that $k\ge
2$, implies that $2K_1\vee 2K_1\cong C_4\trianglelefteq G$, and thus $\lambda_2(G)\ge\lambda_2(C_4)=0$, a
contradiction. Thus, $|G_i|=1$ for some $i$ with $1\le i\le k$, say
$i=1$.

Since $\lambda_2\left(\vee_{i=1}^{n}K_1\right)=\lambda_2(K_n)=-1\not\in(-1,-0.5858)$,
we have $|G_j|\ge 2$ for some $j$ with $2\le j\le k$, say $j=2$. By
Claim~$1$, $G_2$ is disconnected in $G$. If $k\ge 3$, then
$K_4-e\trianglelefteq G$, and thus
$\lambda_2(G)\ge\lambda_2(K_4-e)=\frac{3-\sqrt{17}}{2}>-0.5858$, a contradiction. It
follows that $k=2$, i.e., $G\cong G_1\vee G_2=K_1\vee G_2$, where $G_2$
is a disconnected subgraph of $G$.

If $G_2$ has at least five components, then $K_1\vee 5K_1\cong S_6\trianglelefteq G$, and thus
$\lambda_2(G)\ge\lambda_2(S_6)=4-\sqrt{21}>-0.5858$, a contradiction.
If $G_2$ has a component which is not complete, then there exists a vertex $u$ in this component which has two
nonadjacent neighbors $v$ and $w$, and thus the subgraph of $G$
induced by $V(G_1)\cup\{u,v,w\}$ is isomorphic to $K_4-e$, implying
that $K_4-e\trianglelefteq G$, also a contradiction. Thus
$G\cong K_1\vee(K_{n_1}\cup\dots\cup K_{n_r})$, where $2\le r\le 4$ and
$n_i\ge 1$ for each $i$ with $1\le i\le r$. Now the result follows from Lemmas~\ref{Kn12-Kn123}
and~\ref{Kn1234}.
\end{Proof}

\bigskip

\noindent {\bf Acknowledgement.} This work was supported by  the
Specialized Research Fund for the Doctoral Program of Higher
Education of China (No.~20124407110002) and the National Natural
Science Foundation of China (No.~11071089).

\end{document}